\documentclass[reqno]{amsart}
\usepackage{hyperref}
\newcommand{\N}{{\mathbb{N}}}
\newcommand{\R}{{\mathbb{R}}}
\newcommand{\M}{{\mathbb{M}}}
\begin{document}
\title[Euler-Lagrange]{Euler-Lagrange equation for a delay variational problem}
\author[Blot  Kon\'e]{Jo\"el Blot and Mamadou I. Kon\'e}
\address{Jo\"el Blot: Laboratoire SAMM UE 4543, \newline
Universit\'e Paris 1 Panth\'eon-Sorbonne, centre P.M.F.,\newline
90 rue de Tolbiac, 75634 Paris cedex 13, France.}
\email{blot@univ-paris1.fr}
\address{Mamadou I. Kon\'e: Laboratoire SAMM UE 4543, \newline
Universit\'e Paris 1 Panth\'eon-Sorbonne, centre P.M.F.,\newline
90 rue de Tolbiac, 75634 Paris cedex 13, France.}
\email{mamadou.kone@malix.univ-paris1.fr}
\date{March 28, 2017}
\begin{abstract}
We establish Euler-Lagrange equations for a problem of Calculus of Variations where the unknown variable contains a term of delay on a segment.
\end{abstract}
\maketitle
\numberwithin{equation}{section}
\newtheorem{theorem}{Theorem}[section]
\newtheorem{lemma}[theorem]{Lemma}
\newtheorem{example}[theorem]{Example}
\newtheorem{remark}[theorem]{Remark}
\newtheorem{definition}[theorem]{Definition}
\newtheorem{proposition}[theorem]{Proposition}
{\bf MSC2010-AMS:}  49K99, 34K38.
\section{Introduction}
We consider the following problem of Calculus of Variations
\[
(P)
\left\{
\begin{array}{cl}
{\rm Minimize} & J(x) := \int_0^T F(t,x_t, x'(t)) dt\\
{\rm when} & x \in C^0([- r,T], \R^n) \\
\null & x_{\mid_{[0,T]}} \in C^1([0,T], \R^n)\\
\null & x_0 = \psi, x(T) = \zeta.
\end{array}
\right.
\]
where $r, T \in (0, + \infty)$, $r < T$, $F : [0,T] \times C^0([ -r, 0], \R^n)\times \R^n \rightarrow \R$ is a functional, $\psi \in C^0([ -r,0], \R^n)$, $\zeta \in R^n$, and $x_t(\theta) := x(t + \theta)$ when $\theta \in [ -r, 0]$ and $t \in [0,T]$. $C^0$ denotes the continuity and $C^1$ denotes the continuous differentiability.
\vskip1mm
The aim of this paper is to establish a first-order necessary condition of optimality for problem $(P)$ which is analogous to the Euler-Lagrange equation of the variational problem without delay.
Note that in other settings of delay variational problems, the question of the establishment of an Euler-lagrange equation was studied, for instance in \cite{HVL} (see references therein), \cite{Hu}, \cite{AB}.
\vskip1mm
Now we describe the contents of the paper. In Section 2, we specify the notation of various functions spaces, we introduce an operator to represent the dual space of $C^0([-r,0], \R^n)$ into a space of bounded variation functions (denoted by $\mathcal{R}_n$) and we establish properties on this operator. In Section 3, we state the main theorem of the paper (Theorem \ref{th31}) on the Euler-Lagrange equation. We provide comments on this theorem. In Section, we introduce function spaces and operators which are specific to the delayed functions and we establish several of their properties. In Section 5, we provide conditions to ensure the Fr\'echet differentiability of the criterion of $(P)$. Section 6 is devoted to the proof of the Theorem \ref{th31}
\section{Notation and recall}
When $X$ and $Y$ are real normed vector spaces, $\mathcal{L}(X,Y)$ is the space of the continuous linear mappings from $X$ into $Y$.When $\Lambda \in \mathcal{L}(X,Y)$, we use the writings $\Lambda \cdot x := \Lambda(x)$, $\langle \Lambda, x \rangle = \Lambda(x)$ when $Y = \R$, and we write the norm of linear continuous operators as  $\Vert \Lambda \Vert_{\mathcal{L}} := \sup \{ \Vert \Lambda \cdot x \Vert_Y : x \in X, \Vert x \Vert_X \leq 1 \}$. The topological dual space of $X$ is denoted by $X^* := \mathcal{L}(X, \R)$.
\vskip1mm
When $a < b$ are two real numbers, the space of the continuous functions from $[a,b]$ into $X$ is denoted by $C^0([a,b], X)$; its norm is $\Vert f \Vert_{ \infty, [a,b]} := \sup \{ \Vert f(t) \Vert_X : t \in [a,b] \}$.
\vskip1mm
When $E$ is a finite-dimensional normed vector space, and when $a<b$ are two real numbers, $BV([a,b], E)$ denotes the space of the bounded variation functions from $[a,b]$ into $E$. $NBV([a,b], E)$ denotes the space of the $g \in BV([a,b], E)$ which are left-continuous on $[a,b)$ and which satisfy $g(a)=0$. When $g \in BV([a,b], E)$, the total variation of $g$ is $V_a^b(g)$ which defined as the supremum of the non negative numbers $\sum_{i=0}^k \Vert f(t_i) - f(t_{i+1}) \Vert_E$ on the set of the finite lists $(t_i)_{0 \leq i \leq k+1}$ such that $a= t_0 < <... < t_{k+1} = b$. The norm on  $NBV([a,b], E)$ is $\Vert g \Vert_{BV} = V_a^b(g)$.
\vskip1mm
Denoting by $\mathcal{B}([a,b])$ the Borel $\sigma$-field of $[a,b]$, when $\gamma \in NBV([a,b], E)$, there exists an unique signed measure $\mu[ \gamma] : \mathcal{B}([a,b]) \rightarrow \R$ such that, for all $\alpha < \beta$ in $[a,b]$, $\mu[ \gamma]([\alpha, \beta)) = \gamma (\beta ) - \gamma (\alpha)$. Necessarily we have $\mu[ \gamma]([ \alpha, \beta]) = \gamma (\beta +) - \gamma( \alpha )$, and when $\beta = b$, $\gamma (b+) := \gamma (b)$.
The Lebesgue-Stieltjes integral build on $\gamma$ is defined by $\int_{\alpha}^{\beta} d \gamma (\theta) f( \theta) := \int_{[\alpha, \beta]} f(\theta) d \mu[\gamma](\theta)$ where $\alpha < \beta$ in $[a,b]$ and where $f$ is $\mu[\gamma]$-integrable. We also recall the useful inequality $\vert \int_a^b d \gamma( \theta) \varphi(\theta) \vert \leq V_a^b( \gamma) \Vert \varphi \Vert_{\infty, [a,b]}$.
\vskip1mm
We denote by $(e_k)_{1 \leq k \leq n}$ the canonical basis of $\R^n$ and by $(e_k^*)_{1 \leq k \leq n}$ its dual basis. When $g \in NBV([a,b], \R^{n*})$, $g(\theta) = \sum_{=1}^n g_k(\theta) e_k^*$, when $f : [a,b] \rightarrow \R^n$, $f(\theta) = \sum_{k=1}^n f^k(\theta) e_k$, where the $f^k$ are $\mu[g_k]$-integrable, we set
\begin{equation}\label{eq21}
\int_{\alpha}^{\beta} dg(\theta) \cdot f(\theta) = \sum_{k=1}^n \int_{\alpha}^{\beta} dg_k(\theta) f^k(\theta).
\end{equation}
The theorem of representation of F. Riesz of $C^0([-r, 0], \R)^*$ permits to define the operator
$$\mathcal{R}_1 :C^0([-r, 0], \R)^* \rightarrow NBV([-r, 0], \R)$$
by
\begin{equation}\label{eq22}
\langle {\ell}, \varphi \rangle = \int_{-r}^0 d \mathcal{R}_1({\ell})(\theta) \varphi(\theta).
\end{equation}
when ${\ell} \in C^0([-r, 0], \R)^*$ and $\varphi \in C^0([-r, 0], \R)$. \\
$\mathcal{R}_1$ is a topological linear isomorphism from $ C^0([-r, 0], \R)^*$ into $NBV([-r, 0], \R)$, and it is an isometry: $\Vert \mathcal{R}_1({\ell}) \Vert_{BV} = \Vert {\ell}\Vert_{ \mathcal{L}}$ when ${\ell} \in C^0([-r, 0], \R)^*$.
\vskip1mm
When $n \in \N$, $n \geq 2$, when $L \in C^0([-r, 0], \R^n)^*$, for all $k \in \{1,...,n \}$ we define ${\ell}_k \in C^0([-r, 0], \R)^*$ by setting 
\begin{equation}\label{eq23}
\langle {\ell}_k, \varphi \rangle := \langle L, \varphi e_k \rangle
\end{equation}
where $\varphi \in C^0([-r, 0], \R)$. We set
\begin{equation}\label{eq24}
\mathcal{R}_n(L) := \sum_{k=1}^n \mathcal{R}_1( {\ell}_k) e_k^*.
\end{equation}
This formula defines an operator
$$\mathcal{R}_n : C^0([-r, 0], \R^n)^* \rightarrow NBV([-r,0], \R^{n*}).$$
When $\phi = \sum_{k=1}^n \phi^k e_k \in C^0([-r,0], \R^n)$ we have 
\[
\begin{array}{ccl}
\langle L, \phi \rangle&=& \langle L,\sum_{k=1}^n \phi^k e_k \rangle = \sum_{k=1}^n \langle L,\phi^k e_k \rangle \\
\null & = & \sum_{k=1}^n \langle {\ell}_k, \phi^k \rangle =  \sum_{k=1}^n \int_{-r}^0 d \mathcal{R}_1({\ell}_k)(\theta) \phi^k(\theta)\\
\null & = & \sum_{k=1}^n \int_{-r}^0 d (\mathcal{R}_1({\ell}_k)(\theta) e_k^*) \cdot \phi(\theta)
\end{array}
\]
and using (\ref{eq24}) we obtain
\begin{equation}\label{eq25}
\langle L, \phi \rangle = \int_{-r}^0 d \mathcal{R}_n(L)(\theta) \cdot \phi( \theta).
\end{equation}
\begin{lemma}\label{lem21}$\mathcal{R}_n$ is a linear topological isomorphism from $C^0([-r,0], \R^n)^*$ onto $NBV([-r,0], \R^{n*})$.
\end{lemma}
\vskip1mm
\begin{proof}
$\R^n$ is endowed by the norm $\Vert \sum_{k=1}^n u^k e_k \Vert := \max_{1 \leq k \leq n} \vert u^k \vert$, and  $\R^{n*}$ is endowed by the norm $\Vert \sum_{k=1}^n p_k e_k^* \Vert := \sum_{k=1}^n \vert p_k \vert$.
\vskip1mm
\noindent
Let $g \in NBV([-r, 0], \R^{n*})$, $g(\theta) = \sum_{k=1}^n g_k(\theta) e_k^*$, with $g_k \in NBV([-r, 0], \R)$. We define the linear functional
$$L^g : C^0([-r,0], \R^n) \rightarrow \R$$
$$\langle L^g, \phi \rangle := \sum_{k=1}^n \int_{-r}^0 dg_k(\theta) \phi^k(\theta) =: \int_{-r}^0 dg (\theta) \cdot \phi(\theta)$$
where $\phi \in  C^0([-r,0], \R^n)$, $\phi(\theta) = \sum_{k=1}^n \phi^k(\theta) e_k$.
\vskip1mm
\noindent
Since $V_{-r}^0(g_k) \leq V_{-r}^0(g)$ for all $k \in \{1,...,n \}$, for all $\phi \in  C^0([-r,0], \R^n)$, we have $\vert \langle L^g, \phi \rangle \vert \leq \sum_{k=1}^n \vert \int_{-r}^0 dg_k(\theta) \phi^k(\theta)\vert \leq \sum_{k=1}^n( \Vert g_k \Vert_{BV} \Vert \phi^k \Vert_{\infty, [-r,0]}$ which implies the following inequality
\begin{equation}\label{eq26}
\vert \langle L^g, \phi \rangle \vert \leq n \Vert g \Vert_{BV} \Vert \phi \Vert_{\infty, [-r,0]}.
\end{equation}
This inequality proves that $L^g \in \mathcal{L}( C^0([-r,0], \R^n), \R)$
\vskip1mm
Hence we can build the linear operator 
$$\mathfrak{L} : NBV([-r, 0], \R^{n*})\rightarrow   C^0([-r,0], \R^n)^*, \;\;\; \mathfrak{L}(g) := L^g.$$
From (\ref{eq26}) we have $\Vert  \mathfrak{L}(g) \Vert_{\mathcal{L}} \leq n \Vert g \Vert_{BV}$ which implies the continuity of $\mathfrak{L}$.
\vskip1mm
\noindent
When $\mathfrak{L}(g) = 0$, for all $\varphi \in C^0([-r,0], \R)$, taking $\phi = \varphi e_k$, we obtain that $\int_{-r}^0 dg_k(\theta) \varphi (\theta) = 0$, therefore $\mathcal{R}_1^{-1}(g_k) = 0$, and since $\mathcal{R}_1$ is a linear isomorphism, we obtain $g_k = 0$ for all $k \in \{1,...,n \}$, hence $g=0$. We have proven that $\mathfrak{L}$ is injective.
\vskip1mm
\noindent
When $L \in   C^0([-r,0], \R^n)^*$, setting $g(\theta) := \sum_{k=1}^n \mathcal{R}_1^{-1}({\ell}_k)(\theta) e_k^*$, we verify that $g \in NBV([-r, 0], \R^{n*})$and that $L = \mathfrak{L}(g)$, and so we have proven that $\mathfrak{L}$ is surjective. Hence $\mathfrak{L}$ is linear bijective and continuous. Using the Inverse mapping Theorem of Banach, we obtain that $\mathfrak{L}^{-1}$ is continuous, and since $\mathcal{R}_n = \mathfrak{L}^{-1}$, we obtain the announced result.
\end{proof}
\vskip1mm
Now we consider a case with a dependence with respect to the time.
\vskip1mm
\begin{theorem}\label{th22} Let $[t \mapsto L(t)] \in C^0([0,T],   C^0([-r,0], \R^n)^*)$. Then the following assertions hold.
\begin{itemize}
\item[(i)] $[t \mapsto \mathcal{R}_n(L(t))] \in C^0([0,T], NBV([-r, 0], \R^{n*}))$
\item[(ii)] $[(t, \theta) \mapsto \mathcal{R}_n(L(t))(\theta)]$ is Lebesgue measurable on $[0,T] \times [-r,0]$
\item[(iii)] $[(t, \theta) \mapsto \mathcal{R}_n(L(t))(\theta)]$ is Riemann integrable on $[0,T] \times [-r,0]$.
\end{itemize}
\end{theorem}
\vskip1mm
\noindent
\begin{proof} Assertion (i) is a straightforward consequence of Lemma \ref{lem21}. Assertions (ii) and (iii) are is proven in \cite{BK} (Theorem 4.1) in the case where $NBV([ -r,0], \M_n(\R))$ is the space of the functions in $BV([ -r,0], \M_n(\R))$ ($\M_n(\R)$ being the space of the real $n \times n$ matrices) which are left-continuous on $(0,T]$ and equal to 0 at $T$. The modifications to do to adapt the proof to the case of the present paper are clear.
\end{proof} 
\vskip1mm
We need the two following results to study the Nemytskii (or superposition) operators.
\begin{lemma}\label{lem23}
Let $\mathcal{E}$, $\mathcal{F}$ be two metric spaces, and $\Phi \in C^0(\mathcal{E}, \mathcal{F})$. $\mathcal{P}_c(\mathcal{E})$ denotes the set of the compacts subsets of $\mathcal{E}$. Then we have:\\
$\forall K \in \mathcal{P}_c(\mathcal{E})$, $\forall \epsilon > 0$, $\exists \delta^{\epsilon} > 0$, $\forall x \in K$, $\forall z \in \mathcal{E}$, 
$d(x,z) \leq \delta^{\epsilon} \Longrightarrow d(\Phi(x), \Phi(z)) \leq \epsilon$.
\end{lemma}
This result is established in \cite{Sc}, p. 355. It permits to compensate the lack for compact neighborhood of compact subset in non locally compact metric spaces, for instance in infinite-dimensional normed spaces. In \cite{BCr} and in \cite{BCNP} we have called it "Lemma of Heine-Schwartz".
\begin{lemma}\label{lem24}
Let $\mathcal{E}$, $\mathcal{F}$ be two metric spaces, $A$ be a nonempty compact metric space, and $\Phi : A \times \mathcal{E} \rightarrow \mathcal{F}$ be a mapping. Then the two following assertions are equivalent.
\begin{itemize}
\item[(i)] $\Phi \in C^0(A \times  \mathcal{E}, \mathcal{F})$.
\item[(ii)] $N_{\Phi} \in C^0(C^0(A, \mathcal{E}), C^0(A, \mathcal{F}))$ where $N_{\Phi}(u) := [ a \mapsto \Phi(a, u(a))]$.
\end{itemize}
\end{lemma}
This result is established in \cite{BCNP} (Lemma 8.10).
\vskip2mm
We need to use the following classical Lemma of Dubois-Reymond.
\begin{lemma}\label{lem25}
Let $\alpha < \beta$ be two real numbers. Let $p \in C^0([ \alpha, \beta ], \R^{n*})$ and \\
 $q \in C^0([ \alpha, \beta ], \R^{n*})$. We assume that, for all $h \in C^1([ \alpha, \beta ], \R^n)$ such that $h(\alpha) = h (\beta) = 0$, we have $\int_{\alpha}^{\beta} (p(t) \cdot h(t) + q(t) \cdot h'(t)) dt = 0$.\\
 Then we have $q \in C^1([ \alpha, \beta ], \R^{n*})$ and $q' = p$.
\end{lemma}
This result is proven in \cite{ATF} (p. 60) when $n = 1$. Working on coordinates, the extension to an arbiratry positive integer number is easy.
\section{The main result}
In this section we state the theorem on the Euler-Lagrange equation as a first-order necessary condition of optimality for problem $(P)$. First we give assumptions which are useful to this theorem.
\begin{itemize}
\item[(A1)] $F \in C^0([0,T] \times C^0([ -r, 0], \R^n) \times \R^n, \R)$.
\item[(A2)] For all $(t, \phi, v) \in [0,T] \times C^0([ -r, 0], \R^n) \times \R^n$, the partial Fr\'echet differential with respect to the second (function) variable, $D_2F(t, \phi,v)$, exists and $D_2F \in C^0([0,T] \times C^0([ -r, 0], \R^n) \times \R^n, C^0([ -r,0], \R^n)^*)$.
\item[(A3)] For all $(t, \phi, v) \in [0,T] \times C^0([ -r, 0], \R^n) \times \R^n$, the partial Fr\'echet differential with respect to the third (vector) variable, $D_3F(t, \phi,v)$, exists and $D_3F \in C^0([0,T] \times C^0([ -r, 0], \R^n) \times \R^n, \R^{n*})$.
\end{itemize}
\vskip3mm
\begin{theorem}\label{th31}
Under (A1, A2, A3) let $x$ be a local solution of the problem (P). 
Then the function $[t \mapsto D_3F(t,x_t,x'(t)) - \int_t^{\min\{t+r, T \}} \mathcal{R}_n(D_2F(s,x_s,x'(s))(t-s) ds ]$ is of class $C^1$ on $[0, T]$ , and we have 
$$
\left\{
\begin{array}{l}
\frac{d}{dt}[D_3F(t,x_t,x'(t))  =\\
 \mathcal{R}_n(D_2F(t,x_t,x'(t)))(0) + \frac{d}{dt}\int_t^{\min\{t+r, T \}} \mathcal{R}_n(D_2F(s,x_s,x'(s))(t-s) ds.
\end{array}
\right.
$$

\end{theorem}
\vskip3mm
\noindent
The operator $\mathcal{R}_n$ which is used in this theorem is defined in Section 2 (formulas (\ref{eq24}), (\ref{eq25})). The Euler-Lagrange equation  of this theorem can be written under the integral form as follows
\[
\left\{
\begin{array}{cl}
D_3F(t,x_t,x'(t))  =& \int_0^t \mathcal{R}_n(D_2F(s,x_s,x'(s)))(0) ds\\
\null &  +\int_t^{\min\{t+r, T \}} \mathcal{R}_n(D_2F(s,x_s,x'(s))(t-s) ds + c
\end{array}
\right.
\]
where $c \in \R^{n*}$ is a constant which is independent of $t$.
\vskip1mm
Note the presence of an advance (the contrary of the delay) in this equation. In other settings, \cite{AB} and \cite{Hu}, the Euler-Lagrange also contains a term of advance. 
\section{A function space and operators}
We define the following function space
\begin{equation}\label{eq41}
\mathfrak{X} := \{ x \in C^0([ -r, T], \R^n) : x_{\mid_{[0,T]}} \in C^1([0,T], \R^n) \}.
\end{equation}
On $\mathfrak{X}$ we consider the following norm
\begin{equation}\label{eq42}
\Vert x \Vert_{\mathfrak{X}} := \sup_{-r \leq t \leq T} \Vert x(t) \Vert + \sup_{0 \leq t \leq T} \Vert x'(t) \Vert.
\end{equation}
\begin{lemma}\label{lem41}
$(\mathfrak{X}, \Vert \cdot \Vert_{\mathfrak{X}})$ is a Banach space.
\end{lemma}
\begin{proof}
We can also write $\Vert x \Vert_{\mathfrak{X}} = \Vert x \Vert_{\infty, [-r,T]} + \Vert x' \Vert_{\infty, [0,T]}$. Since $\Vert \cdot \Vert_{\infty, [-r,T]}$ and $\Vert \cdot \Vert_{\infty, [0,T]}$ are norms, $\Vert \cdot \Vert_{\mathfrak{X}}$ is a norm. We consider the space $C^1([0,T], \R^n)$ endowed with the norm $\Vert x \Vert_{C^1, [0,T]} := \Vert x \Vert_{\infty, [0,T]} + \Vert x' \Vert_{\infty, [0,T]}$. We know that $(C^1([0,T], \R^n), \Vert \cdot \Vert_{C^1, [0,T]})$ is a Banach space. Let $(x^k)_{k \in \N}$ be a Cauchy sequence in $(\mathfrak{X}, \Vert \cdot \Vert_{\mathfrak{X}})$. Since $(x^k_{\mid_{|0,T]}})_{k \in \N}$ is also a Cauchy sequence in $C^1([0,T], \R^n)$ there exists $u \in C^1([0,T], \R^n)$ such that $\lim_{k \rightarrow + \infty} \Vert x^k_{\mid_{[0,T]}} - u \Vert_{C^1, [0,T]} = 0$. Since $(x^k)_{k \in \N}$ is also a Cauchy sequence in the Banach space $(C^0([-r,T], \R^n), \Vert \cdot \Vert_{\infty, [0,T]})$ there exists $v \in C^0([-r,T], \R^n)$ such that $\lim_{k \rightarrow + \infty} \Vert x^k -v \Vert_{\infty, [-r,T]} = 0$.\\
Since $\Vert \cdot \Vert_{\infty, [0,T]} \leq \Vert \cdot \Vert_{C^1, [0,T]}$ we have $\lim_{k \rightarrow + \infty} \Vert x^k_{\mid_{[0,T]}} - u \Vert_{\infty, [0,T]} = 0$, and since $\Vert \cdot \Vert_{\infty, [0,T]} \leq \Vert \cdot \Vert_{\infty, [-r,T]}$ we have $\lim_{k \rightarrow + \infty} \Vert x^k_{\mid_{[0,T]}} - v_{\mid_{[0,T]}} \Vert_{\infty, [0,T]} = 0$. Using the uniqueness of the limit we obtain $v_{\mid_{[0,T]}} = u$. Therefore we have $v \in \mathfrak{X}$ and from the inequality $\Vert x^k - u \Vert_{\mathfrak{X}} \leq \Vert x^k - v \Vert_{\infty, [-r,T]} + \Vert x^k - v \Vert_{C^1, [0,T]}$ we obtain $\lim_{ k \rightarrow + \infty} \Vert x^k - u \Vert_{\mathfrak{X}}= 0$. 
\end{proof}
We define the set
\begin{equation}\label{eq43}
\mathfrak{A} := \{ x \in \mathfrak{X} : x_0 = \psi, x(T) = \zeta \}.
\end{equation}
\begin{lemma}\label{lem42}
$\mathfrak{A}$ is a non empty closed affine subset of $\mathfrak{X}$ and the unique vector subspace which is  parallel to $\mathfrak{A}$ is $\mathfrak{V} := \{ h \in \mathfrak{X} : h_0 = 0, h(T) = 0 \}$.
\end{lemma}
\begin{proof}
Setting $y(t) := \psi(t)$ when $t \in [-r, 0]$ and $y(t) := \frac{t}{T}(\zeta - \psi(0)) + \psi(0)$, we see that $y \in \mathfrak{A}$ which proves that $\mathfrak{A}$ is nonempty. From the inequalities $\Vert x(T) \Vert \leq \Vert x \Vert_{\mathfrak{X}}$ and $\Vert x_{\mid_{[-r,0]}} \Vert_{\infty, [-r,0]}  \leq  \Vert x \Vert_{\mathfrak{X}}$, we obtain that $\mathfrak{A}$ is closed in $\mathfrak{X}$. It is easy to verify that $\mathfrak{A}$ is affine. The unique vector subspace of $\mathfrak{X}$ is $\mathfrak{V} = \mathfrak{A} - u$ where $u \in \mathfrak{A}$. and we can easily verify the announced formula for $\mathfrak{V}$.
\end{proof}
When $x \in C^0([ -r,T], \R^n)$ we define
\begin{equation}\label{eq44}
\underline{x} : [0,T] \rightarrow C^0([ -r,0], \R^n), \hskip3mm \underline{x}(t) := x_t.
\end{equation}
\begin{lemma}\label{lem43}
When $x \in C^0([ -r, T], \R^n)$ we have $\underline{x} \in C^0([0,T], C^0([ -r,0], \R^n))$.
\end{lemma}
\begin{proof}
Using a Heine's theorem, since $[-r,T]$ is compact and $x$ is continuous, $x$ is uniformly continuous on $[-r,T]$, i.e.\\
$\forall \epsilon > 0, \exists \delta_{\epsilon} > 0, \forall t,s \in [-r,T], \vert t-s \vert \leq \delta_{\epsilon} \Longrightarrow \Vert x(t) - x(s) \Vert \leq \epsilon$.\\
Let $\epsilon > 0$; if $t, s \in [0,T]$ are such that $\vert t-s \vert \leq \delta_{\epsilon} $ then, for all $\theta \in [-r,0]$ we have $\vert (t + \theta) - (s + \theta) \vert \leq \delta_{\epsilon}$ which implies $\Vert x(t + \theta) - x(s + \theta) \Vert \leq \epsilon$, therefore $\Vert \underline{x} (t) - \underline{x}(s) \Vert_{\infty, [0,T]} \leq \epsilon$.
\end{proof}
After Lemma \ref{lem43} we can define the operator
\begin{equation}\label{eq45}
\mathcal{S} : C^0([ -r, T], \R^n) \rightarrow C^0([0,T], C^0([ -r,0], \R^n)), \hskip3mm \mathcal{S}(x) := \underline{x}.
\end{equation}
\begin{lemma}\label{lem44}
$\mathcal{S}$ is a linear continuous operator from $ (C^0([ -r, T], \R^n), \Vert \cdot \Vert_{\infty})$ into $( C^0([0,T], C^0([ -r,0], \R^n)), \Vert \cdot \Vert_{\infty})$. Setting $\mathcal{S}^1 := \mathcal{S}_{\mid_{\mathfrak{X}}}$, $\mathcal{S}^1$ is a linear continuous operator from $(\mathfrak{X}, \Vert \cdot \Vert_{\mathfrak{X}})$ into $( C^0([0,T], C^0([ -r,0], \R^n)), \Vert \cdot \Vert_{\infty})$.
\end{lemma}
\begin{proof}
The linearity of $\mathcal{S}$ is clear. When $x \in C^0([-r,T], \R^n)$ we have $\Vert \mathcal{S}(x) \Vert_{\infty} = \sup_{0 \leq t \leq T}(\sup_{-r \leq \theta \leq 0} \Vert x(t + \theta) \Vert) = \sup_{-r \leq s \leq T} \Vert x(s) \Vert = \Vert x \Vert_{\infty, [-r,T]}$ which implies the continuity of $\mathcal{S}$. \\
The continuity of $\mathcal{S}^1$ results from the inequality  $\Vert \cdot \Vert_{\infty, [-r,T]} \leq \Vert \cdot \Vert_{\mathfrak{X}}$.
\end{proof}
Now we consider the following operator 
\begin{equation}\label{eq46}
\mathcal{D} : \mathfrak{X} \rightarrow C^0([0,T], \R^n), \hskip3mm \mathcal{D}(x) := x'.
\end{equation}
\begin{lemma}\label{lem45}
The operator $\mathcal{D}$ is linear continuous from $(\mathfrak{X}, \Vert \cdot \Vert_{\mathfrak{X}})$ into \\
$(C^0([0,T], \R^n), \Vert \cdot \Vert_{\infty})$.
\end{lemma}
\begin{proof}
The linearity of $\mathcal{D}$ is clear. When $x \in \mathfrak{X}$, we have
$$\Vert \mathcal{D}(x) \Vert_{\infty, [0,T]} = \Vert x' \Vert_{\infty, [0,T]} \leq \Vert x \Vert_{\mathfrak{X}}$$
which implies the continuity of $\mathcal{D}$.
\end{proof}
When $V$ and $W$ are normed vector spaces we consider the operator 
$$B : \mathcal{L}(V,W) \times E \rightarrow W, B(L,y) := L \cdot y.$$
$B$ is bilinear continuous, and when $I$ is a compact interval of $\R$, we consider the Nemytskii operator defined on $B$
\begin{equation}\label{eq47}
\left.
\begin{array}{r}
N_B : C^0(I, \mathcal{L}(V,W)) \times C^0(I,V) \rightarrow C^0(I,W)\\
N_B(L, h) := [t \mapsto B(L(t), h(t)) = L(t) \cdot h(t) ]
\end{array}
\right\}
\end{equation}
where we have assimilated $ C^0(I, \mathcal{L}(V,W)) \times C^0(I,V)$ and $ C^0(I, \mathcal{L}(V,W)) \times V)$. $N_B$ is bilinear and the following inequality holds
\begin{equation}\label{eq48}
\forall L \in C^0(I, \mathcal{L}(V,W)), \forall h \in C^0(I,V), \Vert N_B(L,h) \Vert_{\infty, I} \leq \Vert L \Vert_{\infty, I} \cdot \Vert h \Vert_{\infty, I}.
\end{equation}
This inequality shows that $N_B$ is continuous and consequently it is of class $C^1$.
\section{The differentiability of the criterion}
First we establish a general result on the differentiability of the Nemytskii operators.
\begin{lemma}\label{lem51}
Let $I$ be a compact interval of $\R$, $V$, $W$ be two normed vector spaces, and $\Phi : I \times V \rightarrow  W$ be a mapping. We assume that the following conditions are fulfilled.
\begin{itemize}
\item[(a)] $\Phi \in C^0(I \times V, W)$.
\item[(b)] For all $t \in I$, the partial Fr\'echet differential of $\Phi$ with respect to the second variable, $D_2 \Phi(t,x)$, exists for all $x \in V$, and $D_2 \Phi \in C^0(I \times V, \mathfrak{L}(V,W))$.
\end{itemize}
Then the operator $N_{\Phi}$ defined by $N_{\Phi}(v) := [ t \mapsto \Phi(t, v(t))]$ is of class $C^1$  from $C^0(I,V)$ into $C^0(I,W)$, and we have $D N_{\Phi}(v) \cdot \delta v = [t \mapsto D_2 \Phi(t, v(t)) \cdot \delta v(t)]$.
\end{lemma}
\begin{proof}
Under our assumptions, from Lemma \ref{lem24} the following assertions hold.
\begin{equation}\label{eq51}
N_{\Phi} \in C^0(C^0(I,V), C^0(I,W))
\end{equation}
\begin{equation}\label{eq52}
N_{D_2 \Phi} \in C^0(C^0(I,V), C^0(I,\mathcal{L}(V,W)).
\end{equation}
We arbitrarily fix $v \in C^0(I,V)$. The set $K := \{ (t,v(t)) : t \in I \}$ is compact as the image of a compact by a continuous mapping. 
Let $\epsilon > 0$; using Lemma \ref{lem24} we have 
\[
\left\{
\begin{array}{l}
\exists \beta^{\epsilon} > 0, \forall t \in I, \forall s \in I, \forall y \in V, \\
\vert t - s \vert + \Vert v(t) - y \Vert \leq \beta^{\epsilon} \Longrightarrow \Vert D_2 \Phi(t,u(t)) - D_2 \Phi(s, y) \Vert \leq \epsilon,
\end{array}
\right.
\]
which implies
$$\exists \beta^{\epsilon} > 0, \forall t \in I,\forall y \in V,\Vert v(t) - y \Vert \leq \beta^{\epsilon} \Longrightarrow \Vert D_2 \Phi(t,u(t)) - D_2 \Phi(t, y) \Vert \leq \epsilon,$$
Let $\delta v \in C^0(I,V)$ such that $\Vert \delta v \Vert_{\infty} \leq \beta^{\epsilon}$. For all $y \in \; ] v(t), v(t) + \delta v(t) [\;  = \{ (1- \lambda) v(t) + \lambda (v(t) + \delta v(t)) \}$, we have $\Vert y \Vert \leq \Vert \delta v(t) \Vert \leq \beta^{\epsilon}$, and consequently $ \Vert D_2 \Phi(t,u(t)) - D_2 \Phi(t, y) \Vert \leq \epsilon$. Using the mean value theorem (\cite{ATF}, Corollaire 1, p. 141 ), we have 
\[
\left\{
\begin{array}{l}
\Vert \Phi(t, v(t) + \delta v(t)) - \Phi(t, v(t)) - D_2 \Phi(t, v(t)) \cdot \delta v(t) \Vert 
\leq \\
 \sup_{y \in ] v(t), v(t)+ \delta v(t) [} \Vert D_2 \phi(t,v(t)) - D_2 \Phi(t,y) \Vert \cdot \Vert \delta v(t) \Vert 
\leq \epsilon \Vert \delta v(t) \Vert
\end{array}
\right.
\]
which implies, taking the supremum on the $t \in I$, 
$$\Vert N_{\Phi}(v + \delta v) - N_{\Phi}(v) - N_B(N_{D_2 \Phi}(v), \delta v ) \Vert_{\infty,I} \leq \epsilon \Vert \delta v \Vert_{\infty,I}.$$
And so we have proven that $N_{\Phi}$ is Fr\'echet differentiable at $v$ and
$$DN_{\Phi}(v) \cdot \delta v = N_B( N_{D_2 \Phi} , \delta v).$$
When $v, v^1, \delta v \in C^0(I, V)$, using (\ref{eq48}) we have 
\[
\begin{array}{l}
\Vert (D N_{\Phi}(v) - D N_{\Phi}(v^1)) \cdot \delta v \Vert_{\infty,I} = \Vert N_B( N_{D_2 \Phi}(v), \delta v) - N_B( N_{D_2 \Phi}(v^1), \delta v)\Vert_{\infty,I} = \\
\Vert N_B(N_{D_2 \Phi}(v) -N_{D_2 \Phi}(v^1), \delta v) \Vert_{\infty,I} \leq \Vert N_{D_2 \Phi}(v) -N_{D_2 \Phi}(v^1) \Vert \cdot \Vert \delta v \Vert_{\infty,I},
\end{array}
\]
and taking the supremum on the $\delta v \in C^0(I,V)$ such that $\Vert \delta v \Vert_{\infty,I} \leq 1$ we obtain 
$$\Vert D N_{\Phi}(v) - D N_{\Phi}(v^1) \Vert_{\infty,I} \leq \Vert N_{D_2 \Phi}(v) -N_{D_2 \Phi}(v^1) \Vert_{\infty,I}$$
and (\ref{eq52}) implies the continuity of $D N_{\Phi}$.
\end{proof}
In different frameworks, similar results of differentiability of Nemytskii operators were proven in \cite{BCNP} (for almost periodic functions) , in \cite{BCr} (for bounded sequences), in \cite{BBC}(for continuous functions which converge to zero at infinite).
\vskip2mm
From $F : [0,T] \times C^0([ -r,0], \R^n) \times \R^n \rightarrow \R$ we define the following Nemytskii operator
\begin{equation}\label{eq53}
\left.
\begin{array}{l}
N_F : C^0([0,T], C^0([ -r, 0], \R^n)) \times C^0([0,T], \R^n) \rightarrow C^0([0,T], \R^n)\\
N_F(U,v) := [ t \mapsto F(t, U(t), v(t))].
\end{array}
\right\}
\end{equation}
\begin{lemma}\label{lem52}
Under (A1, A2, A3), $N_F$ is of class $C^1$ and for all $U$, $ \delta U \in  C^0([0,T], C^0([ -r, 0], \R^n))$, for all $v 
$, $\delta v \in C^0([0,T], \R^n)$ we have \\
$D N_F(U,v) \cdot(\delta U, \delta v) = [t \mapsto D_2F(t,U(t),v(t)) \cdot \delta U(t) + D_3 F(t, U(t), v(t)) \cdot \delta v(t)]$.
\end{lemma}
\begin{proof}
It is a straightforward consequence of Lemma \ref{lem51} with $V = C^0([ -r, 0], \R^n) \times \R^n$, $W = \R$, $\Phi = F$, and by using that the differential of $F(t,\cdot , \cdot)$ at $(U(t), v(t))$ applied to $(\delta U(t), \delta v(t))$ is equal to $D_2F(t,U(t),v(t)) \cdot \delta U(t) + D_3 F(t, U(t), v(t)) \cdot \delta v(t)$.
\end{proof}
\begin{lemma}\label{lem53}
Under (A1, A2, A3), $J \in C^1(\mathfrak{X}, \R)$ and for all $x \in \mathfrak{A}$ and for all $h \in \mathfrak{V}$, we have\\
$DJ(x) \cdot h = \int_0^T (D_2 F(t, x_t, x'(t)) \cdot h_t + D_3F(t, x_t, x'(t)) \cdot h'(t)) dt$.
\end{lemma}
\begin{proof}
We introduce the operator $in : \mathfrak{X} \rightarrow C^0([-r,T], \R^n)$ by setting $in (x) := x$, and the functional $I : C^0([0,T], \R) \rightarrow \R$ by setting $I(f) := \int_0^T f(t)dt$ the Riemann integral of $f$ on $[0,T]$. The operator $in$ is clearly linear and from the inequality $\Vert \cdot \Vert_{\infty, [-r,T]} \leq \Vert \cdot \Vert_{\mathfrak{X}}$, it is continuous. $I$ is linear and by using the mean value theorem, it is continuous.\\
Note that $J = I \circ N_F \circ (\mathcal{S} \circ in, \mathcal{D})$. Since $in$, $\mathcal{S}$, $\mathcal{D}$ and $I$ are linear continuous, they are of class $C^1$, and so $(\mathcal{S} \circ in, \mathcal{D})$ is of class $C^1$. Using Lemma \ref{lem52}, $N_F$ is of class $C^1$, and so $J$ is of class $C^1$ as a composition of $C^1$ mappings. The calculation of $DJ$ is a simple application of the Chain Rule :
\[
\begin{array}{ccl}
DJ(x) \cdot h &=& DI(N_F(\mathcal{S}(in(x), \mathcal{D}(x)) \cdot DN_F\mathcal{S}(in(x), \mathcal{D}(x)).\\
\null & \null & (D \mathcal{S}(in(x)\cdot D in (x)h, D \mathcal{D}(x) \cdot h)\\
\null & = & I(D N_F( \underline{x}, x'). (\underline{h}, h'))\\
\null & = & \int_0^T(D_2F(t, x_t,x'(t)) \cdot h_t + D_3 F(t, x_t,x'(t)) \cdot h'(t))dt.
\end{array}
\] 
\end{proof}
\section{Proof of the main result}
To abridge the writing, we write $D_2F[t] := D_2F(t, x_t, x'(t))$ and $D_3F[t] := D_3F(t, x_t, x'(t))$, and in the proofs we write $g(t, \theta) := \mathcal{R}_n(D_2F[t])(\theta)$.
\vskip1mm
\begin{lemma}\label{lem61} Under (A1, A2, A3),
for all $h \in \mathfrak{V}$, we have
\[
\left\{
\begin{array}{l}
\int_0^T D_2F[t] \cdot h_t dt = \\
\int_0^T \mathcal{R}_n(D_2F[t])(0) \cdot h(t) dt + \int_0^T \left( \int_{t}^{\min \{t+r, T \}} \mathcal{R}_n(D_2F[s](t-s)ds \right) \cdot h'(t) dt.
\end{array}
\right.
\]
\end{lemma} 
\begin{proof}
Using Proposition 3.2 in \cite{BK} and $g(t, -r) = 0$, we have, for all $t \in [0,T]$, 
\[
\begin{array}{rcl}
D_2F[t] \cdot h_t  & = & \int_{-r}^0 d_{\theta} g(t,\theta) \cdot h(t+ \theta)\\
\null &=& \int_{t-r}^t d_{\xi} g(t, \xi -t) \cdot h(\xi)\\
\null &=& g(t,0) \cdot h(t)  - \int_{t-r}^t g(t, \xi -t) \cdot h'(\xi) d \xi,
\end{array}
\]
which implies 
\begin{equation}\label{eq61}
\int_0^T D_2F[t] \cdot h_t dt = \int_0^T g(t,0) \cdot h(t) dt - \int_0^T \int_{t-r}^t g(t, \xi -t) \cdot h'(\xi) d \xi dt.
\end{equation}
We set $A := \{ (t,\xi) : 0 \leq t \leq T, t-r \leq \xi \leq t \}$ and from the Fubini-Tonelli theorem we have 
\begin{equation}\label{eq62}
\int_0^T \int_{t-r}^t g(t, \xi -t) \cdot h'(\xi) d \xi dt = \int \int_A g(t, \xi -t) d \xi dt
\end{equation}
\vskip2mm
\noindent
For each $\xi$, we consider $A_{., \xi} := \{ t \in [0,T] : (t, \xi) \in A \}$. We have 
$$
A_{., \xi} = \left\{
\begin{array}{cccl}
\null & [0, \xi + r] & {\rm if} & \xi \in [-r,0] \\ 
\null & [\xi, \xi + r] & {\rm if} & \xi \in [0, T-r] \\
\null & [\xi, T] & {\rm if} & \xi \in [T-r, T].
\end{array}
\right.
$$
Using the Fubini theorem, we obtain
\begin{equation}\label{eq63}
\left.
\begin{array}{ccl}
\int \int_A g(t, \xi -t) h'(\xi) d \xi & = & \int_{-r}^T ( \int_{A_{., \xi}}  g(t, \xi -t) h'(\xi)dt) d \xi\\
\null & = & \int_{-r}^0 ( \int_{A_{., \xi}}  g(t, \xi -t) h'(\xi)dt) d \xi \\
\null & \null & +\int_0^{T-r} ( \int_{A_{., \xi}}  g(t, \xi -t) h'(\xi)dt) d \xi \\
\null & \null &+  \int_{T-r}^T  ( \int_{A_{., \xi}}  g(t, \xi -t) h'(\xi)dt) d \xi.
\end{array}
\right\}
\end{equation}
For the first term of (\ref{eq63}), since $h$ is equal to zero on $[-r,0]$, we have $h'$ equal to zero on $[-r,0]$ and consequently we obtain
\begin{equation}\label{eq64}
 \int_{-r}^0 ( \int_{A_{., \xi}}  g(t, \xi -t) h'(\xi)dt) d \xi = 0.
\end{equation}
For the second term of (\ref{eq63}), we have
$$
\int_0^{T-r} ( \int_{A_{., \xi}}  g(t, \xi -t) h'(\xi)dt) d \xi = \int_0^{T-r} ( \int_{\xi}^{\xi + r} g(t, \xi -t) h'(\xi)dt) d \xi$$
 and replacing $\xi$ by $t$ and $t$ by $s$ we obtain
\begin{equation}\label{eq65}
\int_0^{T-r} ( \int_{A_{., \xi}}  g(t, \xi -t) h'(\xi)dt) d \xi = \int_0^{T-r} ( \int_t^{t+r} g(s, t-s)ds) \cdot h'(t)dt.
\end{equation}
For the third term of (\ref{eq63}) we have
\[
\begin{array}{crcl}
\null & \int_{T-r}^T ( \int_{A_{.,s}} g(t, s-t) \cdot h'(s) dt) ds &=& \int_{T-r}^T ( \int_{A_{.,s}} g(t, s-t)dt) \cdot h'(s) ds\\
=&\int_{T-r}^T ( \int_s^T g(t, s-t)dt) \cdot h'(s) ds & = & \int_{T-r}^T ( \int_s^T g(\alpha, \beta - \alpha)d \alpha) \cdot h'(\beta) d \beta\
\end{array}
\]
which implies
\begin{equation}\label{eq66}
\int_{T-r}^T ( \int_{A_{.,s}} g(t, s-t) \cdot h'(s) dt) ds = \int_{T -r}^T ( \int_t^T g(s, t-s) ds ) \cdot h'(t) dt.
\end{equation}
Using (\ref{eq64}), (\ref{eq65}) and (\ref{eq66}) in (\ref{eq63}) we obtain  
\begin{equation}
\int_0^T ( \int_{t-r}^t g(t, s-t) \cdot h'(s) ds) dt = \int_O^T( \int_t^{\min \{ t+r, T \}} g(s, t-s) ds ) \cdot h'(t) dt.
\end{equation}
Using (\ref{eq66}) in (\ref{eq61}) we obtain the announced formula.
\end{proof}
We set 
\[
\left\{
\begin{array}{ccl}
p(t) & := &\mathcal{R}_n(D_2F(t, x_t, x'(t)))(0)\\
q(t) & := & D_3F(t, x_t, x'(t)) - \int_t^{\min \{ t+r, T \}} \mathcal{R}_n(D_2 F(s, x_s, x'(s)))(t-s) ds.
\end{array}
\right.
\]
We know that $x$ is a local minimizer of $J$ on the closed affine subset $\mathfrak{U}$, that $\mathfrak{V}$ is the tangent vector subspace of $\mathfrak{U}$ at $x$ after Lemma \ref{lem42}. From Lemma \ref{lem53}, we know that $J$ is of class $C^1$, and then, using a classical argument, we can assert that $DJ(x) \cdot h = 0$ for all $h \in \mathfrak{V}$. Using Lemma \ref{lem53}, we obtain

$$0 = D J(x) \cdot h = \int_o^T (p(t) \cdot h(t) + q(t) \cdot h'(t)) dt$$
and so, using Lemma \ref{lem25}, we obtain tat $q$ is $C^1$ on $[0,T]$ and that $q' = p$ which is the formula given in the statement of Theorem \ref{th31}. Hence Theorem \ref{th31} is proven.
\end{document}